\documentclass[leqno,12pt]{article}

\usepackage{amsmath, amsfonts, amssymb}
\usepackage{mathrsfs}
\usepackage{bm}
\newtheorem{thm}{Theorem}[section]
\newtheorem{prop}[thm]{Proposition}
\newtheorem{lem}[thm]{Lemma}

\newtheorem{defn}[thm]{Definition}

\newcommand{\ra}{\rightarrow}
\newcommand{\dis}{\displaystyle}

\def\veps{\varepsilon}
\def\Ad{\mathrm{Ad}}

\def\ad{\textrm{ad}}

\def\G{\mathscr G}
\def\lg{\mathcal{L}_eG}
\def\lc{\mathcal{L}}
\def\pg{\mathcal{P}_eG}
\def\pg1{\mathcal{P}_e^1 (G)}

\def\pc{\mathcal{P}}
\def\p{\mathscr{P}}

\def\R{\mathbb R}
\def\Z{\mathbb Z}
\def\N{\mathbb N}

\def\C{\mathscr C}

\def\d{\text{\rm{d}}}

\def\la{\langle}
\def\raa{\rangle}

\def\cut{\mathit{Cut}}

\newcommand{\dd}[2]{d_{L^2}(#1,#2)}

\newcommand{\fin}{\hspace*{\fill}\rule{0.3em}{1ex}}
\newenvironment{proof}{{\bf \noindent Proof.}}{\fin}

\numberwithin{equation}{section}
\begin{document}

\title{The existence of geodesics in Wasserstein spaces over path groups and loop groups}
\author{Jinghai Shao \footnote{Email: shaojh@bnu.edu.cn}\\[0.6cm]  Center for Applied Mathematics, Tianjin University, Tianjin 300072, China}
\maketitle
\begin{center}
\begin{minipage}{12cm}
In this work we prove the existence and uniqueness of the optimal transport map for  $L^p$-Wasserstein distance with $p>1$, and particularly present an explicit expression of the optimal transport map for the case $p=2$. As an application, we show the existence of geodesics connecting probability measures satisfying suitable condition on path groups and loop groups.
\end{minipage}
\end{center}

\noindent\textbf{Keywords:} Monge-Kantorovich problem; Path groups; Loop groups; Heat kernel measure;
Pinned Wiener measure


\section{Introduction}

  In the seminal works of K.T. Sturm \cite{St} and Lott-Villani \cite{LV}, a new concept of curvature-dimension condition has been developed on the abstract metric space to replace the lower bound of Ricci curvature of Riemannian manifold via the convexity of the relative entropy on the Wasserstein space. This convexity is measured by the behavior of the relative entropy along  geodesics connecting two probability measures in the Wasserstein space over this metric space. This concept is equivalent to the Ricci curvature lower bound for Riemannian manifold as shown in \cite{RS} and possesses the advantage of stability under Gromov-Hausdorff convergence. There are many extensions of this concept in various setting, for example, Finsler space \cite{Oh}, Alexandrov spaces \cite{Pe, ZZ}, infinitesimally Hilbertian metric measure spaces \cite{EKS}. The starting point of this concept is that the studied Wasserstein space is a geodesic space, that is, for any two probability measures $\nu_0$ and $\nu_1$ satisfying some additional condition, there exists a geodesic under the $L^2$-Wasserstein metric. The validation of this basic property usually depends on the study of the Monge-Kantorovich problem in respective space.
  In this work, we shall study the optimal transport on path groups and loop groups and apply the obtained optimal transport maps to show the existence of geodesics in the Wasserstein spaces over path groups and loop groups.

  The Monge-Kantorovich problem is to consider how to move the mass
  from one distribution to another as efficiently as possible. Here
  the efficiency is measured against a positive cost function
  $c(x,y)$.
  Precisely, given two probability measures $\mu$ and $\nu$ on a measurable space $X$,
  define its Wasserstein distance by
  \begin{equation}\label{1.1}W_c(\mu,\nu)=\inf\Big\{\int_{X\times X} c(x,y)\,\pi(\d x,\d y);\quad\pi\in\C(\mu,\nu)\Big\},
  \end{equation}
  where $c:X\times X\ra[0,+\infty]$ is called the cost function and $\C(\mu,\nu)$ is the set of
  all probability measures on $X\times X$ with marginals $\mu$ and
  $\nu$ respectively. Then the Monge-Kantorovich problem is to find
  a measurable map $\mathscr T$ satisfying $\nu=(\mathscr T)_\ast \mu$ such that the probability measure $\pi=(id\times \mathscr
  T)_\ast \mu$ attains the infimum in (\ref{1.1}). Here the notion
  $(\mathscr T)_\ast \mu$ denote the push forward of measure $\mu$
  by a measurable map $\mathscr T$, i.e. $(\mathscr
  T)_\ast\mu=\mu\circ \mathscr T^{-1}$; $id$ denotes the identity
  map.
  It is well known that the solving of this problem is very
  crucially dependent on the cost function. On Euclidean space $\R^d$ and Riemannian manifold,
  there are many works to solve this problem with respect to different
  cost functions such as \cite{Bre} \cite{Mc} \cite{GM} \cite{KS}. Refer to \cite{Am} for general
  survey on this respect and to \cite{Vi} for detail discussions.

  When the dimension of the space goes to infinity, Feyel and
  \"Ust\"unel in \cite{FU} proved the existence and uniqueness of the optimal
  transport map on the abstract Wiener space.
  In \cite{FS}, together with Fang, we solved the Monge-Kantorovich problem on loop groups.  There we use the ``Riemannian distance", a kind of Cameron-Martin
  distance in some sense, to define the $L^2$-Wasserstein distance. The advantage of this distance is that there exists a  sequence of
  suitable finite dimensional approximations, which makes it possible to use the results in finite dimensional Lie groups. However,
  the ``Riemannian distance" is too large. It behaves like the Cameron-Martin distance in Wiener space in some sense, which equals
  to infinite almost everywhere with respect to the Wiener measure. This causes great difficulty in ensuring the finiteness of the Wasserstein distance between two probability measures on loop groups. Furthermore, there is no explicit expression of the optimal transport in this case.

  In this work, we shall use another important distance, $L^2$-distance, to define the Wasserstein distance on path or loop groups. Since the $L^2$-distance is always bounded when the Lie group is  compact, the induced Wasserstein distance between any two
  probability measures is finite. Therefore, the finiteness of Wasserstein distance is no longer a constraint of the existence of optimal transport map in this situation. As an application,
  there exists an invertible optimal transport map pushing the heat kernel measure forward to the pinned Wiener measure on loop group. These two probability measures play important role in the stochastic analysis of loop groups.  Another advantage of using $L^2$-distance is that an explicit form of optimal transport map can be given, which helps us to show the existence of geodesic connecting two probability measures on path groups or loop groups.

  The existence of optimal transport map has a lot of applications. For example, it is applied to construct the solution of Monge-Amp\`ere equation (cf. for instance, \cite{EG}), and to establish Pr\'ekopa-Leindler inequalities in \cite{CMS}.  In \cite{JKO}, it helps to construct the
  gradient flow of relative entropy in the space of probability measures, which provides a new method to
  construct the solution of Fokker-Planck
  equations.  This method has been systemically studied and was developed to deal with more general differential equations in \cite{AGS}.

  When studying the Monge-Kantorovich problem on path and loop groups using the $L^2$-distance, we need to consider the derivative of Riemannian distance on Lie group, which adds some condition on Lie group about the cut locus of its identity element.

  Let $G$ be a connected compact Lie group with Lie algebra $\G$ which is endowed
  with an inner product $\la\,,\,\raa_\G$, and the associated Riemannian distance is denoted by $\rho(\cdot,\cdot)$. Given a
  point $x\in G$, point $y\in G$ is called a \emph{cut point} of $x$ if there exists a geodesic $\gamma:[0,\infty)\ra G$
  parameterized by arc length with $\gamma(0)=x$ such that $\gamma(t_0)=y$ for some $t_0>0$ and for any $t\leq t_0$,
  $\rho(\gamma(0),\gamma(t))=t$ and for any $t>t_0$, $\rho(\gamma(0),\gamma(t))<t$.
  The union of all cut points of $x$ is called the \emph{cut locus} of $x$ and denoted by $\cut(x)$.
  A map $V:[a,b]\ra \G$ is called a \emph{piecewise continuous curve} if there exists a finite subdivision $a=a_0<a_1<\ldots<a_k=b$ such that
  $V\big|_{[a_{i-1},a_i]}$ is continuous for $i=1,\ldots,k$.

  The condition on the cut locus used in this work is:

 (H)\quad  If the cut locus $\cut(e)$ of the identity element $e$ of $G$ is not empty, then for any continuous curve $\{x_t\}_{t\in [a,b]}\subset \cut(e)$, there exists a piecewise continuous curve $\{X_t\}_{t\in[a,b]}$ in $\G$ such that $\mathrm{exp}_e\, X_t=x_t,\ \forall \,t\in [a,b]$, where $\mathrm{exp}_e$ denotes the exponential map determined by the geodesic equations in the setting of Riemannian manifold.

  \noindent\textbf{Examples}: \begin{itemize}
  \item the $n$-dimensional torus $T_n=S^1\times\cdots\times S^1$ is a connected compact Lie group and satisfies the hypothesis (H).

  \item The Heisenberg group $\mathbb{H}^n$ endowed with Carnot-Carath\'eodory distance satisfies the assumption (H) by \cite[Theorem 3.4]{AR}.
  Indeed, the Heisenberg group $\mathbb{H}^n$ is a noncommutative stratified nilpotent Lie group. As a set it can be identified with its
  Lie algebra $\R^{2n+1}\simeq\mathbb C^n\times \R$ via exponential coordinates. Denote a point in $\mathbb H^n$ by $\bm{x}=(\xi,\eta,t)=[\zeta,t]$
  where $\xi=(\xi_1,\ldots,\xi_n)$, $\eta=(\eta_1,\ldots,\eta_n)\in \R^n$, $t\in \R$ and $\zeta=(\zeta_1,\ldots,\zeta_n)\in \mathbb C^n$ with
  $\zeta_j=\xi_j+i\eta_j$. The group law is given by $[\zeta,t]\cdot [\zeta',t']:=[\zeta+\zeta', t+t'+2\sum_{j=1}^n \mathrm{Im} \zeta_j\bar{\zeta}'_j\,]$. The set $L^\ast:=\{[0,s]\in \mathbb H^n;\ s\in \R\backslash \{0\}\}$ is the cut locus of identity element $[0,0]\in \mathbb H^n$. Set $\mathbb{S}=\{a+ib\in \mathbb{C}^n;\,|a+ib|=1\}$. For any $a+ib\in\mathbb{S}$, $v\in\R$ and $r>0$, we say that a curve $\gamma:[0,r]\ra \mathbb{H}^n$ is a curve with parameter $(a+ib,v,r)$ if $\gamma(s)=(\xi(s),\eta(s),t(s))$ where
  \begin{align*}
    \xi_j(s)&=\frac{r}{v}\Big(b_j\Big(1-\cos\frac{vs}{r}\Big)+a_j\sin\frac{vs}{r}\Big),\\
    \eta_j(s)&=\frac{r}{v}\Big(-a_j\Big(1-\cos\frac{vs}{r}\Big)+b_j\sin\frac{vs}{r}\Big),\\
    t(s)&=\frac{2r^2}{v^2}\Big(\frac{vs}{r}-\sin\frac{vs}{r}\Big),\quad \quad j=1,\ldots,n,
  \end{align*} when $v\neq 0$ and
  \[\gamma(s)=(a_1s,\ldots, a_ns,b_1s,\ldots,b_ns,0)\]
  when $v=0$.
  Each curve with parameter $(a+ib, 2\pi, \sqrt{\pi |t|})$ for some $a+ib\in \mathbb{S}$ is a sub-unit minimal geodesic from $[0,0]$ to $\bm x=[0,t]\in L^\ast$, from which one can easily verify $\mathbb H^n$ satisfies the assumption (H).
  \end{itemize}

  In the following, after introducing some necessary notations on path and loop groups,  we present our main results of this paper.

  Denote $\mathcal P(G)$ the path group,
  that is,
  $$\pc(G)=\big\{\ell:[0,1]\ra G \ \text{continuous};\ \ell(0)=e\big\},$$
  where $e$ denotes the unit element of Lie group $G$.
  Let
  $\rho(\cdot,\cdot)$ be the Riemannian metric on $G$, that is,
  $$\rho(x,y)=\inf \Big\{L(\gamma):=\Big(\int_0^1 \big|\gamma(t)^{-1}\frac{\d}{\d t}\gamma(t)\big|_\G^2\,\d
  t\Big)^{1/2}\Big\},
  $$ where the infimum is taken over all absolutely continuous curves connecting $x$ and $y$. It is easy to see
  $\rho(x,y)=\rho(e,x^{-1}y)$ by the definition. The topology of
  $\pc(G)$ is determined by the uniform
  distance $d_\infty(\gamma_1,\gamma_2)$ for $\gamma_1,\,\gamma_2\in \pc(G)$, i.e.
  \begin{equation}\label{1.2}d_\infty(\gamma_1,\gamma_2):=\max_{t\in [0,1]}\rho(\gamma_1(t),\gamma_2(t)).
  \end{equation}
  Under this topology, $\pc(G)$ becomes a complete
  separable space. We now introduce another distance on $\pc(G)$, the $L^2$-distance:
  \begin{equation}\label{1.3}
  d_{L^2}(\gamma_1,\gamma_2)=\Big(\int_0^1\rho(\gamma_1(t),\gamma_2(t))^2\,\d t\Big)^{1/2}.
  \end{equation} It is obvious that $d_{L^2}(\gamma_1,\gamma_2)\leq
  d_\infty(\gamma_1,\gamma_2)$ for any $\gamma_1,\,\gamma_2\in \pc(G)$. In this paper, we consider
  the Wasserstein distance induced by the $L^2$-distance on $\pc(G)$. Given two probability
  measures $\nu$ and $\sigma$ over $\pc(G)$, the $L^p$-Wasserstein
  distance  between them is defined by:
  \begin{equation}\label{1.4}
  W_p(\nu,\sigma)=\inf\Big\{\int_{\pc(G)\times\pc(G)}\dd
  {\gamma_1}{\gamma_2}^p\,\pi(\d \gamma_1,\d \gamma_2);\ \pi\in
  \C(\nu,\sigma)\Big\}^{1/p}, \ \ p>1,
  \end{equation}
  where $\C(\nu,\sigma)$ stands for the set of all probability
  measures on the product space $\pc(G)\times \pc(G)$ with marginals
  $\nu$ and $\sigma$ respectively.

  Set $\mu$ be the Wiener measure on
  $\pc(G)$, which is the diffusion measure corresponding to the left
  invariant Laplace operator $\frac 12\sum_{i=1}^d\tilde{\xi}_i^2$ on
  $G$, where $\{\xi_1,\ldots,\xi_d\}$ denotes an orthonormal basis of
  $\G$ and $\tilde{\xi}$ denotes the associated left invariant
  vector field on $G$.

  Our first main results are the following two theorems on the existence and uniqueness of optimal transport maps on path groups and loop groups.
  \begin{thm}\label{t1.1} Let $G$ be a connected compact Lie group and satisfy assumption (H). Let $\nu$ and $\sigma$ be two probability measures on $\pc(G)$,
  and assume $\nu$ is absolutely continuous with respect to the
  Wiener measure $\mu$ on $\pc(G)$. Then for each $p>1$, there
  exists a unique measurable map $\mathscr T_p:\pc(G)\ra\pc(G)$ such
  that it pushes $\nu$ forward to $\sigma$ and
  \begin{equation*}
  W_p(\nu,\sigma)^p=\int_{\pc(G)}\dd{\gamma}{\mathscr
  T_p(\gamma)}^p\,\d\nu(\gamma).
  \end{equation*}
  Furthermore, there exists some function $\phi$ in the Sobolev space $\mathbf D_1^2(\mu)$ such that the map $\mathscr T_2$ can be expressed as
  \begin{equation}\label{opt-e}\mathscr T_2(\gamma)(t)=\exp_{\gamma(t)}\Big(\frac 12 \ell_{\gamma(t)}\frac{\d^2}{\d t^2}\big(\nabla \phi(\gamma)\big)(t)\Big), \ \ a.e.\, t\in[0,1],
  \end{equation}
  for almost every $\gamma\in \pc(G)$. Here $\exp_\gamma$ denotes the geodesic exponential map on Lie group.
  \end{thm}

  \begin{thm}\label{t1.2} Let $G$ be a connected compact Lie group and satisfy assumption (H).
  Let $\lg=\{\ell:[0,1]\ra G\
  \text{continuous};\,\ell(0)=\ell(1)=e\}$.  Let $\sigma_1$ and $\sigma_2$ be two probability measures on $\lg$. Assume $\sigma_1$ is
  absolutely continuous with respect to the heat kernel measure $\nu$ on $\lg$.
  Then for each $p>1$ there exists a unique measurable map $\mathscr T_p:\lg\ra\lg$
  such that $(\mathscr T_p)_\ast \sigma_1=\sigma_2$ and
  $$ W_p(\sigma_1,\sigma_2)^p=\int_{\lg} \dd{\ell}{\mathscr T_p(\ell)}^p\,\d\sigma_1(\ell),$$
  where $$ W_p(\sigma_1,\sigma_2)^p:=
  \inf\Big\{\int_{\lg\times\lg}\dd{\ell_1}{\ell_2}^p\,\pi(\d\ell_1,\d\ell_2);\ \pi\in \C(\sigma_1,\sigma_2)\Big\}.$$
  In particular, for each $p>1$, there exists a unique measurable map $\mathscr
  T_p:\lg\ra\lg$ such that $\mathscr T_p$ pushes heat kernel measure
  $\nu$ forward to pinned Wiener measure $\mu_0$ on $\lg$, and its
  inverse $\mathscr T_p^{-1}$ pushes $\mu_0$ forward to $\nu$.

  Moreover, for $p=2$ there exists some $\phi$ in the Sobolev space $ \mathbf{D}_1^2(\nu)$ such that the map $\mathscr T_2$ can be expressed as
  \begin{equation}\label{opt-2}\mathscr T_2(\gamma)(\theta)=\exp_{\gamma(\theta)}\Big(\frac 12 \ell_{\gamma(\theta)}\frac{\d^2}{\d \theta^2}\big(\nabla \phi(\gamma)\big)(\theta)\Big), \ \ a.e.\, \theta\in[0,1],
  \end{equation}
  for almost every $\gamma\in \lg$. Here $\exp_\gamma$ denotes the geodesic exponential map on Lie group.
  \end{thm}

The basic idea to prove Theorem \ref{t1.1} and Theorem \ref{t1.2} is similar to that of \cite{FS,Mc}  based on the solution of dual Kantorovich problem. The solution of dual Kantorovich problem gives us a pair of functions $(\phi,\phi^c)$, where
\[\phi^c(y)=\inf_{x\in X}\{c(x,y)-\phi(x)\},\]
for some cost function $c(\cdot,\cdot)$ on some metric space $X$. Then the key point is to show that there is a uniquely determined measurable map $y=\mathscr T(x)$ such that
\[\phi(x)+\phi^c(\mathscr T(x))=c(x,\mathscr T(x))\]
holds for suitable choice of $x$. In the language of $c$-convexity (cf. \cite[Chapter 5]{Vi}), it is equivalent to show that the subdifferential $\partial_c \phi(x)$ contains only one element for suitable choice of $x$.

Due to the explicit expression of the optimal transport map for $L^2$-Wasserstein distance, we applied previous results to show the existence of geodesics in the Wasserstein spaces over path groups and loop groups.

  \begin{thm}\label{t1.3}
    Assume the conditions of Theorem \ref{t1.1} hold. Then for any two probability measures $\nu_0,\,\nu_1$ on $\pc(G)$ with $\nu_0$  being absolutely continuous w.r.t. Wiener measure $\mu$, there exists a curve of probability measures $(\nu_r)_{r\in[0,1]}$ connecting $\nu_0$ and $\nu_1$ satisfying
    \[W_2(\nu_0,\nu_r)=rW_2(\nu_0,\nu_1),\quad r\in[0,1].\]

    Similarly, under the conditions of Theorem \ref{t1.2}, for any two probability measures $\sigma_0,\,\sigma_1$ on $\lg$ with $\sigma_0$ absolutely continuous w.r.t. the heat kernel measure, there exists a geodesic $(\sigma_r)_{r\in[0,1]}$ in $(\mathscr P(\lg), W_2)$ connecting $\sigma_0$ to $\sigma_1$.
  \end{thm}

   This paper is organized as follows: in the next section, we introduce some notations and basic results on Lie group. In section 3, we
   give the proofs of Theorem \ref{t1.1} in the case $p=2$ and Theorem \ref{t1.3} in order
   to explain the idea of the argument. For the general case
   $p>1$, the proof of Theorem \ref{t1.1} is stated in section 4. In the last section, we
   investigate the Monge-Kantorovich problem on loop groups. Some basic
   notations on loop group and the argument of Theorem \ref{t1.2}
   are stated there.

\section{Preliminaries}

   We first review some basic notions and results on the Lie group and
   its Lie algebra. The proofs of these results will be omitted, and refer to Warner's book \cite{War} for details.

   A Lie group $G$ is a differentiable manifold which is also endowed
   with a group structure such that the map $G\times G\ra G$ defined by
   $(\sigma,\tau)\mapsto \sigma \tau^{-1}$ is smooth.
   Let $\sigma\in G$, left translation by $\sigma$ and right
   translation by $\sigma$ are respectively the diffeomorphisms
   $\ell_\sigma$ and $r_\sigma$ of $G$ defined by $$
   \ell_\sigma(\tau)=\sigma \tau,\quad r_\sigma(\tau)=\tau\sigma \quad
   \text{for all $\tau\in G$}.$$
   A vector field $X$ on $G$ is called left invariant if for each
   $\sigma\in G$, $$d\ell_\sigma\circ X=X\circ \ell_\sigma.$$

   A Lie algebra of the Lie group $G$ is defined to be the Lie algebra
   $\G$ of left invariant vector fields on $G$. The map $\alpha:\G\ra T_eG$ defined by $\alpha(X)=X(e)$ is an
   isomorphism from the Lie algebra $\G$ to the tangent space of $G$ at
   the identity.  $\alpha $ is injective and surjective. It will be convenient at times
   to look on the Lie algebra as the tangent space of $G$ at the identity. We consider the left invariant vector fields on $G$, then the tangent space $T_gG$ at every point $g\in G$ can be viewed as $g\G$, and the inner product $\la\,,\, \raa_\G$ induces a inner product on $T_gG$ by
   \[\la d\ell_g\circ X,d\ell_g\circ Y\raa=\la X, Y\raa_\G,\quad X,\,Y\in \G.\]

\noindent\textbf{Examples of Lie group and its Lie algebra:}
\begin{itemize}
\item[a)]The set $ gl(n,\R)$ of all
$n\times n$ real matrices is a real vector space. Matrices are added
and multiplied by scalars componentwise. $gl(n,\R)$ becomes a Lie
algebra if we set $[A,B]=AB-BA$.

\textit{The general linear group} $Gl(n,\R)$ is the set of all
$n\times n$ non-singular real matrices. Then $Gl(n,\R)$ becomes a
Lie group under  matrix multiplication, and $gl(n,\R)$ can be
considered as the Lie algebra of $Gl(n,\R)$.

\item[b)] \textit{Special linear group} $Sl(n,\R)=\{A\in
Gl(n,\R):\mathrm{det}\, A=1\}$ is a Lie group. Its Lie algebra will
be matrices of trace $0$, $sl(n,\R)=\{A\in gl(n,\R):
\mathrm{trace}\, A=0\}$.
\end{itemize}

\begin{defn}\label{b5}
Let $G$ and $H$ be Lie groups. A map $\phi:G\ra H$ is a (Lie group)
homomorphism if $\phi$ is both $C^\infty$ and a group homomorphism
of the abstract groups.

Let $\G$ and $\mathscr{H}$ be Lie algebra, a map $\psi:\G\ra
\mathscr{H}$ is a (Lie algebra) homomorphism if it is linear and
preserves Lie brackets, i.e. $\psi([X,Y])=[\psi(X),\psi(Y)]$ for all
$X,Y\in \G$.
\end{defn}

A homomorphism $\phi:\R\ra G$ is called a 1-parameter subgroup of
$G$.  For each $X\in \G$, there exists a unique 1-parameter
subgroup $t\mapsto \sigma_X(t)$ such that its tangent vector at $0$ is $X(e)$.
This induces a definition of exponential map on Lie group by $\exp X=\sigma_X(1)$.
This definition of exponential map does not depend on the metric on
$\G$. In matrix Lie groups, the exponential map $\exp A$
coincides with the usual exponential of matrices
$$\exp A=\sum_{k=0}^\infty \frac {A^k}{k!}.$$

\begin{defn}\label{b10}
Let $\sigma\in G$. We define the action $\Ad_\sigma$ on $\G$ by
\begin{equation}\label{ad1}
\Ad_\sigma X=\Big \{\frac{\d}{\d t}
\sigma\exp(tX)\sigma^{-1}\Big\}_{t=0} \quad\text{for each X in
$\G$}.
\end{equation}
$\Ad:G\ra \mathrm{Aut}(\G)$ is called the adjoint representation of
Lie group $G$, where $\mathrm{Aut}(\G)$ denotes the set of
automorphisms on $\G$.
\end{defn}
By the definition, it is easy to obtain that for all $\sigma,\,\tau\in G$  $\Ad_\sigma X=d\ell_\sigma d r_{\sigma^{-1}} X$.
Let $X,\,Y\in \G$, define
\begin{equation}\label{ad2}
\ad_X Y=\Big\{\frac{\d}{\d t} \Ad_{\exp tX} Y\Big\}_{t=0}.
\end{equation}
Then $\dis\ad_X Y=[X,Y].$

For each $A\in \G$, let $\tilde{A}$ denote the unique left invariant
vector field on $G$ determined by $A$. Given a metric $\la\ ,\, \raa$
on Lie algebra $\G$.  It can induce a left invariant Riemannian
metric on $G$.  The Levi-Civita connection on $G$ induced by this
metric is given by
\begin{equation}\label{b13}
\la \nabla_{\tilde{A}} \tilde{B},\tilde{C}\raa= \frac 12\big\{\la
[A,B],C\raa-\la [A, C],B\raa-\la [B,C],A\raa\big\},\quad \text{for
$A,\ B,\ C\in \G$}.
\end{equation}
Let $\ad_A^\ast$ be the adjoint operator of $\ad_A$ w.r.t $\la\ ,\,
\raa$. Then
\begin{equation}\label{b13.5}
\nabla_{\tilde{A}} \tilde{B}=\widetilde{\nabla_A B},\quad \nabla_A
B=\frac 12\big(\ad_A B-\ad_A^\ast B-\ad_B^\ast A\big).
\end{equation}
Given an orthonormal basis $\{e_i\}_{i=1}^d$ of $\G$, since $\la
\tilde{e}_i,\tilde{e}_j\raa_\sigma=\la e_i,e_j\raa$,
$\{\tilde{e}_i\}_{i=1}^d$ is a family of orthonormal vector fields
on $G$. Let
\begin{equation}\label{b14}
\Gamma_{ij}^k=\la \nabla_{e_i} e_j,e_k\raa=\la
\nabla_{\tilde{e}_i}\tilde{e}_j,\tilde{e}_k\raa.
\end{equation}
Then with respect to the Levi-Civita connection a $C^1$ curve
$(\gamma_t,\ a<t<b)$ on $G$ is called a \textit{geodesic} if
$\dot{\gamma}(t):=\d \gamma(t)/\d t$ is parallel along $\gamma$,
i.e.
\begin{equation}\label{b15}
\nabla_{\dot{\gamma}(t)} \dot{\gamma}(t)=0.
\end{equation}
Setting $\dot{\gamma}_k(t)=\la \dot{\gamma}(t),\tilde{e}_k\raa$,
then the equation (\ref{b15}) turns into
\begin{equation}\label{b16}
\frac {\d\dot{\gamma}_k(t)}{\d t}+\sum_{i,j=1}^d \Gamma_{ij}^k
\dot{\gamma}_i(t)\dot{\gamma}_j(t)=0,\quad \text{for
$k=1,\ldots,d$}.
\end{equation}
Note that in this equation, the Christoffel coefficients
$\Gamma_{ij}^k$ are independent of the curve $\gamma(t)$, which is
different to general geodesic equations on manifolds. This geodesic equation induces another definition of exponential
map on Lie group when being viewed as a Riemannian manifold, and this exponential map depends on the inner product defined on Lie
algebra $\G$. But when $\G$ is endowed with an $\Ad$-invariant metric $\la\ ,\ \raa$, namely, $\la \Ad_g X,\Ad_g Y\raa=\la X, Y\raa$, for all
$g\in G$ and $X,\,Y\in \G=T_eG$, then 1-parameter subgroups are geodesics (see \cite[Corollary 3.19]{CE}) and every geodesic is
coincident with a translation of a segment of 1-parameter subgroup (see J. F. Price \cite[Theorem 4.3.3]{Pri}). This enable us to know that
for compact connected Lie groups the exponential maps induced by 1-parameter subgroup are surjective.
It is known (cf. \cite[proposition 5.4]{CE}) that the cut locus of each point $g$ on $G$ is closed, and contains two kinds of points, i.e. if $g'$ is in the cut locus of $g$, then
$g'$ is either the first conjugate point of $g$ along some geodesic connecting $g$ with $g'$, or there exists at least two minimizing geodesics
joining $g$ to $g'$. When $G$ is a simply connected Lie group with $\Ad$-invariant metric, then all geodesics minimize up to the first conjugate point (cf. \cite[Corollary 5.12]{CE}).

There are lots of work about infinite dimensional stochastic analysis on path groups and loop groups. We refer
to \cite{Dri,DL} and the book \cite{Fang} for some basic facts and results.

  \section{Proof of Theorem \ref{t1.1}: the case $p=2$}
  Let us recall a well known result on the solving of Kantorovich dual
  problem. Refer to \cite[Theorem 5.10]{Vi} for the argument.
  \begin{thm}\label{t2.1}
  Let $X$ and $Y$ be two Polish spaces and $\mu$, $\nu$ be two probability measures on $X$ and $Y$ respectively. Let
  $c:X\times Y\ra \R$ be a lower semicontinuous cost function such that
  $$\forall \, (x,y)\in X\times Y,\quad c(x,y)\geq a(x)+b(y)$$
  for some real-valued upper semicontinuous functions  $a\in L^1(\mu)$ and $b\in L^1(\nu)$.
  Then if $$C(\mu,\nu):=\inf_{\pi\in \C(\mu,\nu)} \int c\,\d\pi$$ is finite, and one has the pointwise upper bound
  \begin{equation}\label{2.1}c(x,y)\leq c_X(x)+c_Y(y),\quad (c_X,c_Y)\in L^1(\mu)\times L^1(\nu),\end{equation}
  then both the primal and dual Kantorovich problems have solutions, so
  \begin{equation}\label{2.2}\begin{split}&\min_{\pi\in\C(\mu,\nu)}\int_{X\times Y} c(x,y)\,\d\pi(x,y)\\
  &=\max_{(\phi,\psi)\in L^1(\mu)\times L^1(\nu):\,\phi+\psi\leq c}\Big(\int_X\phi(x)\,\d\mu(x)+\int_Y\psi(y)\,\d\nu(y)\Big)\\
  &=\max_{\phi\in L^1(\mu)}\Big(\int_X\phi(x)\,\d\mu(x)+\int_Y\phi^c(y)\,\d\nu(y)\Big),
  \end{split}
  \end{equation} where
  \begin{equation}\label{2.3}
  \phi^c(y):=\inf_{x\in X}\big\{ c(x,y)-\phi(x)\big\}.
  \end{equation}
  \end{thm}

  In our situation, the diameter $D$ of Lie group $G$ is
  finite as $G$ is assumed to be compact. Then $\dd xy\leq D<+\infty$, and hence the
  condition (\ref{2.1}) is satisfied for any probability measures
  $\nu$ and $\sigma$ on $\pc(G)$. $\dd xy$ is also continuous on
  $\pc(G)$. According to Theorem \ref{t2.1}, it holds that for any two probability
  measures $\nu$ and $\sigma$ on $\pc(G)$,
  $$W_2(\nu,\sigma)^2=\sup\Big\{\int_{\pc(G)} \phi(x)\,\nu(\d
  x)+\int_{\pc(G)} \psi(y)\,\sigma(\d y)\Big\},$$
  where the supremum runs among all pairs of measurable functions
  $(\phi,\psi)$ such that $\phi(x)+\psi(y)\leq \dd xy^2$.  Moreover, there exists a pair of functions
  $(\psi^c,\psi)$ such that
  \begin{equation}\label{2.4}
  W_2(\nu,\sigma)^2=\int_{\pc(G)}\psi^c(x)\,\nu(\d
  x)+\int_{\pc(G)}\psi(y)\,\sigma(\d y),
  \end{equation} where $\psi^c(x)=\inf_{x\in\pc(G)}\big\{\dd
  xy^2-\psi(y)\big\}$. In the rest of this section, we will fix such pair of functions $(\psi^c,\psi)$ and denote by
  $\phi(x)=\psi^c(x)$.
  Then $\phi$ is Lipschitz continuous with respect to the distance
  $\dd xy$. In fact, for any fixed $x,\,z\in
  \pc(G)$, for any $\veps>0$, there exists  $y_\veps\in \pc(G)$ such
  that $\psi^c(z)\geq \dd z{y_\veps}^2-\psi(y_\veps)-\veps$. Then
  \begin{align*}\phi(x)-\phi(z)&=\psi^c(x)-\psi^c(z)\\
     &\leq \dd x{y_\veps}^2-\psi(y_\veps)-\dd
     z{y_\veps}^2+\psi(y_\veps)+\veps\\
     &\leq \big(\dd x{y_\veps}+\dd z{y_\veps}\big)\big(\dd x{y_\veps}-\dd
     z{y_\veps}\big)+\veps\\
     &\leq 2D\dd xz+\veps.
  \end{align*}
  Letting $\veps\ra 0^+$, we get $\phi(x)-\phi(z)\leq 2D\dd xz$.
  Changing the place of $x$ and $z$, we get  $\phi$ is Lipschitz continuous.

  Next, we shall use the Rademacher's theorem on path group $\pc(G)$
  to show that $\phi$ is in the Sobolev space. Before this, we
  introduce some basic notions. Let
  \begin{equation}\label{2.5}
  H(\G)=\Big\{h:[0,1]\ra \G;\,h(0)=0,\,|h|_H^2=\int_0^1\big|\dot{h}(t)\big|_\G^2\,\d
  t<+\infty\Big\},
  \end{equation} where dot $\cdot$ stands for the derivative with respect
  to $t$.
  Let $F:\pc(G)\ra \R$ be a measurable function. We set
  $$D_h F(\gamma)=\frac{\d}{\d\veps}\Big|_{\veps=0}F(\gamma e^{\veps
  h}),\quad h\in H(\G),\ \gamma\in \pc(G).$$ A function $F\in
  L^2(\mu)$ is said to be in the Sobolev space $\mathbf D_1^2(\mu)$
  if there exists $\nabla F\in L^2(\mu;H(\G))$ such that for each
  $h\in H(\G)$, it holds $D_h F=\la\nabla F,h\raa_{H}$
  in $L^{2-}(\mu)$, where $L^{2-}(\mu)=\bigcap_{p<2}L^p(\mu)$ and
  $$\la h_1,h_2\raa_H=\int_0^1\la
  \dot{h}_1(t),\dot{h}_2(t)\raa_\G\,\d t,\quad  h_1,\,h_2\in H(\G).$$
  A function $F:\pc(G)\ra \R$ is said to be cylindrical if
  $$F(\gamma)=f(\gamma(t_1),\ldots,\gamma(t_n)), \quad  f\in C^\infty(G^n),\ 0<t_1<\ldots<t_n\leq 1, \ n\in \N.$$
  Let $\mathbf{Cyln}(\pc(G))$ denote the space of all cylindrical
  functions. Due to \cite{Ai}, 
  the space $\mathbf{Cyln}(\pc(G))$ is dense in $\mathbf D_1^2(\mu)$.
  Now we introduce the third distance, Cameron-Martin distance
  $d_\pc$, on $\pc(G)$, that is, for $\gamma_1,\,\gamma_2\in
  \pc(G)$,
  \begin{equation}\label{2.6}
  d_{\pc}(\gamma_1,\gamma_2)=\Big(\int_0^1\big|v(t)^{-1}\dot{v}(t)\big|_\G^2\,\d
  t\Big)^{1/2},\quad \text{if}\ v=\gamma_1^{-1}\gamma_2\ \text{absolutely
  continuous};
  \end{equation} otherwise, set $d_\pc(\gamma_1,\gamma_2)=+\infty$.
  It's easy to check that $d_\infty(\gamma_1,\gamma_2)\leq
  d_\pc(\gamma_1,\gamma_2)$ for all $\gamma_1,\,\gamma_2\in \pc(G)$.
  According to the Rademacher's theorem \cite[Theorem 1.5]{Sh} and
  discussions in subsection 2.1 therein, we obtain that
  \begin{lem}\label{t2.2}
  Any bounded
  $d_\pc$-Lipschitz continuous function $F$ on $\pc(G)$ belongs to $\mathbf
  D_1^2(\mu)$.
  \end{lem}
  Here and in the sequel, a function
  $F$ on a metric space $(X, d)$
  is said to be \textbf{$d$-Lipschitz continuous}, where $d$ is a metric on $X$, if
  there exists some constant $C>0$ such that
  $$|F(x)-F(y)|\leq C d(x,y),\quad
  \forall \,x,\,y\in X.$$
  Due to the fact $$\dd{\gamma_1}{\gamma_2}\leq
  d_\infty(\gamma_1,\gamma_2)\leq d_\pc(\gamma_1,\gamma_2),$$
  it is clear that a $d_{L^2}$-Lipschitz continuous function $F$ is also
  $d_\infty$-Lipschitz
  and $d_\pc$-Lipschitz continuous. Combining $d_{L^2}$-Lipschitz continuity of $\phi$ with
  Lemma \ref{t2.2}, we get
  $\phi$ is in the Sobolev space $\mathbf{D}_1^2(\mu)$.

  \begin{prop}[Key proposition]\label{t2.3} If there exist $\gamma_1$ and $\gamma_2$
  such that
  \begin{equation}\label{2.7}
  \phi(\gamma_1)+\psi(\gamma_2)=\dd{\gamma_1}{\gamma_2}^2,
  \end{equation}
  and $\phi$ is differentiable at $\gamma_1$, then $\gamma_2$ is
  uniquely determined by $\gamma_1$ and $\phi$.
  \end{prop}

  \begin{proof}
  For $h\in H(\G)$ and $\veps>0$, by the fact $\phi=\psi^c$, we get
  $$\phi(\gamma_1e^{\veps h})+\psi(\gamma_2)\leq
  \dd{\gamma_1e^{\veps h}}{\gamma_2}^2.$$ Subtracting (\ref{2.7})
  from both sides of this inequality yields
  \begin{equation}\label{2.8}
  \begin{split}
  \phi(\gamma_1e^{\veps h})-\phi(\gamma_1)&\leq \dd{\gamma_1e^{\veps
  h}}{\gamma_2}^2-\dd{\gamma_1}{\gamma_2}^2\\
  &=\int_0^1\rho\big(\gamma_1(t)e^{\veps h(t)},\gamma_2(t)\big)^2\,\d t-\int_0^1\rho\big(\gamma_1(t),\gamma_2(t)\big)^2\,\d t.
  \end{split}
  \end{equation}
  For each fixed $t\in[0,1]$, there exits a constant speed geodesic $v_t:[0,1]\ra G$
  such that $v_t(0)=\gamma_2(t)^{-1}\gamma_1(t)$,
  $v_t(1)=e$ and
  $$L(v_t)^2=\int_0^1\big|v_t^{-1}(s)\frac{\d}{\d
  s}v_t(s)\big|_\G^2\,\d s=\rho(\gamma_1(t),\gamma_2(t))^2.$$
  Set $\tilde v_t(s)=v_t(s)e^{(1-s)\veps h(t)},\ s\in [0,1]$. Then
  $\tilde v_t(0)=\gamma_2(t)^{-1}\gamma_1(t)e^{\veps h(t)}$ and $\tilde
  v_t(1)=e$. Hence,
  \begin{equation}\label{2.9}
  \rho\big(\gamma_1(t)e^{\veps
  h(t)},\gamma_2(t)\big)^2\leq L(\tilde v_t)^2.
  \end{equation}
  As $$d_s\tilde v_t(s)=\big(\dot v_t(s)e^{(1-s)\veps h(t)}-\veps
  \tilde v_t(s)h(t)\big)d s,$$ where $d_s$ stands for the derivative with respect to $s$, we get
  \begin{equation*}\begin{split}
  L(\tilde v_t)^2&=\int_0^1\big|\tilde v_t(s)^{-1}\dot{\tilde
  v}_t(s)\big|_\G^2\,\d s\\
  &=\int_0^1\big|\Ad_{e^{-(1-s)\veps
  h(t)}}v_t(s)^{-1}\dot{v}_t(s)-\veps h(t)\big|_\G^2\,\d s\\
  &=\int_0^1\big| v_t(s)^{-1}\dot v_t(s)\big|_\G^2 -2\veps \la
  \Ad_{e^{-(1-s)\veps h(t)}}v_t(s)^{-1}\dot
  v_t(s),h(t)\raa_\G+\veps^2|h(t)|_\G^2\,\d s\\
  &=\rho(\gamma_1(t),\gamma_2(t))^2- 2\veps\int_0^1 \la
  \Ad_{e^{-(1-s)\veps h(t)}}v_t(s)^{-1}\dot
  v_t(s),h(t)\raa_\G\,\d s+\veps^2 |h(t)|_\G^2.
  \end{split}
  \end{equation*}
  Invoking (\ref{2.8}) and (\ref{2.9}), we obtain
  \begin{equation*}
  \phi(\gamma_1e^{\veps h})-\phi(\gamma_1)\leq
  -2 \veps\int_0^1\int_0^1\la \Ad_{e^{-(1-s)\veps h(t)}}v_t(s)^{-1}\dot
  v_t(s),h(t)\raa_\G\,\d s\d t+\veps^2\int_0^1|h(t)|_\G^2\,\d t.
  \end{equation*}
  Dividing both sides by $\veps$, letting $\veps\ra 0^+$ and $\veps\ra 0^-$ respectively, it
  follows
  \begin{align}\label{2.10}
  &\la \nabla \phi(\gamma_1), \,h\raa_H\leq -2\int_0^1\big\la
  \int_0^1v_t(s)^{-1}\dot v_t(s)\,\d s,\, h(t)\big\raa_\G\,\d t,\\
  \label{2.11}
  &\la \nabla \phi(\gamma_1),\, h\raa_H\geq -2\int_0^1\big\la
  \int_0^1v_t(s)^{-1}\dot v_t(s)\,\d s,\, h(t)\big\raa_\G\,\d t.
  \end{align} Set \begin{equation}\label{2.12}
  V_t(u)=\int_0^u v_t(s)^{-1}\dot v_t(s)\,\d s,\quad u\in [0,1],
  \end{equation}
  then we have shown by (\ref{2.10}) (\ref{2.11}) that
  \begin{equation}\label{2.13}
  \la \nabla \phi(\gamma_1),\,h\raa_H=-2\int_0^1\la
  V_t(1),\,h(t)\raa_\G\,\d t,
  \end{equation}
  which implies that if $V_t(1)$ as a
  function of $t$ is continuous at some $t_0\in [0,1]$ then
  $V_{t_0}(1)$ is uniquely determined. In fact, take a sequence of
  smooth functions $h_\veps$ such that $0\leq h_\veps\leq 1$, and
  $$ h_\veps(t)=\left\{\begin{array}{ll}1&t\in[t_0-\veps, t_0+\veps]\cap[0,1],\\
  0 &t\not\in [t_0-2\veps,t_0+2\veps]\cap [0,1].
  \end{array}\right.$$Set $\{e_1,\ldots,e_d\}$ be an orthonormal  basis of $\G$. We have
  \begin{equation}\label{2.14}
  \la V_{t_0}(1),e_i\raa_\G=\lim_{\veps\ra 0}\int_0^1\la V_t(1),h_\veps(t)e_i\raa_\G\,\d
  t=\lim_{\veps\ra 0}-\frac12\la \nabla
  \phi(\gamma_1),h_\veps e_i\raa_H.
  \end{equation}
  Moreover, as $\la\nabla\phi(\gamma_1),h_\veps e_i\raa_H$ is measurable
  from $\pc(G)$ to $\R$, the limitation $\la V_{t_0}(1),e_i\raa_\G$ is also
  measurable. Then
   $V_{t_0}(1)=\sum_i \la V_{t_0}(1),e_i\raa_\G e_i$ is  measurable with
  respect to the variable $\gamma_1$.

  We shall show below the following assertion: according to our assumption (H), we can always choose a family of minimizing
  geodesics $(v_t)_{t\in [0,1]}$ such that $v_t(0)=\gamma_2^{-1}(t)\gamma_1(t)$, $v_t(1)=e$, and there exists an at most countable subset $\Omega\subset [0,1]$  such that $t\mapsto \dot{v}_t(1)$ is continuous on $[0,1]\backslash \Omega$. Hence, $t\mapsto V_t(1)$ is continuous on
  $[0,1]\backslash \Omega$. If this assertion is correct, then $V_t(1)$ is uniquely determined by $\nabla \phi(\gamma_1)$ at all $t\in [0,1]
  \backslash \Omega$. Then, due to Lemma 3.4 below, we know that $V_t(1)$ determines  uniquely a geodesic $v_t:[0,1]\ra G$ so that $v_t(1)=e$.
  Therefore, $\gamma_2(t)$ is uniquely determined by $\nabla \phi(\gamma_1)$ at $t\in [0,1]\backslash \Omega$. Since $t\mapsto \gamma_2(t)$ is
  continuous, then we get the desired result that $\gamma_2$ is uniquely determined by  $\phi$ and $\gamma_1$. Now we proceed to the proof of previous assertion.

  \noindent\textbf{Case 1:}
  If $\{\gamma_2(t)^{-1}\gamma_1(t),\ t\in [0,1]\}$ does not go across the cut locus of $e$ in
  $G$, then there exists a unique family of minimizing geodesics
  $\big(v_t\big)_{t\in [0,1]}$ such that $v_t(0)=\gamma_2^{-1}(t)\gamma_1(t)$, $v_t(1)=e$, and $t\mapsto\dot{v}_t(1)$ is
  continuous.
  The geodesic equation guarantees that $t\mapsto v_t(s)^{-1}\dot{v}_t(s)$ for $s\in [0,1]$ is also continuous, which implies the continuity of $t\mapsto V_t(1)$ for $t\in [0,1]$.

  \noindent\textbf{Case 2:} If $\gamma_2^{-1}(t)\gamma_1(t)$ goes across the cut locus of $e$. Then the continuity
  of $\gamma_1(t),\,\gamma_2(t)$ and the closeness of the cut locus of $e$
  yield that the set $\Omega$ containing all $t$ such that $\gamma_2^{-1}(t)\gamma_1(t)$
  enters or leaves the cut locus of $e$ is not empty and  at most countable.
  So the set $I_1:=\{t\in [0,1]\backslash\Omega;\ \gamma_2^{-1}(t)\gamma_1(t)\not\in \cut{(e)}\}$ and the set
  $I_2:=\{t\in [0,1]\backslash\Omega;\ \gamma_2^{-1}(t)\gamma_1(t)\in \cut{(e)}\}$ can both be represented as the union of at most countable
  open intervals. For each open interval $(s_1,s_2)\subset I_1$, above discussion in case 1 show that there exist a curve $t\mapsto V_t(1)$ for
  $t\in (s_1,s_2)$. For each open interval $(s_1',s_2')\subset I_2$, our assumption (H) may guarantee that we can choose a family of geodesics $(v_t)_{t\in (s_1',s_2')}$ such that $v_t(0)=\gamma_2^{-1}\gamma_1(t)$, $v_t(1)=e$, and $t\mapsto \dot v_t(0)$ is continuous on $(s_1',s_2')$.
  This yields $t\mapsto V_t(1)$ is continuous on $(s_1',s_2')$.

  In all, we can choose a $(V_t(1))_{t\in [0,1]}$ such that $t\mapsto V_t(1)$ is continuous on $[0,1]\backslash \Omega$. Therefore, we have proved
  the assertion and complete the proof of this proposition.
  \end{proof}

  \begin{lem}\label{t2.4}
  Using the notations as above. Then $V_t(1)$ determines uniquely a
  minimizing geodesic $v_t:[0,1]\ra G$ such that 
  $v_t(1)=e$.
  \end{lem}
  \begin{proof}
  Let $a\in \G,\,\veps\in \R$ and $c\in C^2([0,1],\R)$ such that
  $c(0)=c(1)=0$. Consider $v_{t,\veps}(s)=v_t(s)e^{\veps c(s) a}, \
  s\in [0,1]$. Then $v_{t,\veps}(0)=v_t(0)$ and
  $v_{t,\veps}(1)=v_t(1)$. $$d_sv_{t,\veps}(s)=v_{t,\veps}(s)\Big(\Ad_{e^{-\veps c(s)a}}v_t(s)^{-1}\dot v_t(s)+\veps c'(s)a\Big)d s,$$
  and
  $$L(v_{t,\veps})^2=\int_0^1\Big|\Ad_{e^{-\veps c(s)a}}v_t(s)^{-1}\dot v_t(s)+\veps
  c'(s)a\Big|_\G^2\,\d s.$$
  Since $\veps\mapsto L(v_{t,\veps})^2$ arrives its minimum at
  $\veps=0$, we get
  \begin{align*}
  0=&\frac{\d}{\d \veps} L(v_{t,\veps})^2=2\int_0^1\la
  v_t(s)^{-1}\dot v_t(s),c'(s) a\raa_\G\,\d s\\
  &=2\int_0^1\la V_t'(s),c'(s)a\raa_\G\,\d s\\
  &=2\la V_t(1), c'(1)a\raa_\G-2\la V_t(0),c'(0)
  a\raa_\G-2\int_0^1\la V_t(s),c''(s)a\raa_\G\,\d s.
  \end{align*} This yields
  \begin{equation}\label{2.15}
  \la V_t(1),c'(1)a\raa_\G=\int_0^1\la V_t(s),c''(s)a\raa_\G\,\d s.
  \end{equation}
  Assume $\tilde{v}_t$ be another minimizing geodesic such that
  $\tilde v_t(1)=e$ and $\tilde{V}_t(1)=V_t(1)$, where $\tilde
  V_t(u)=\int_0^u\tilde v_t(s)^{-1}\dot{\tilde v}_t(s)\,\d s$, $u\in
  [0,1]$. Then analogous deduction yields
  $$\int_0^1\la \tilde V_t(s),c''(s)a\raa_\G\,\d s=\int_0^1\la
  V_t(s),c''(s) a\raa_\G\,\d s.$$ Since $s\mapsto c''(s)a$ is dense
  in $L^2(\mu;\G)$,
  $$V_t(s)=\tilde V_t(s),\quad \text{for almost every}\ s\in
  [0,1].$$
  The continuity of $s\mapsto V_t(s)$ and $s\mapsto \tilde V_t(s)$
  yields
  $$V_t(s)=\tilde V_t(s),\quad \text{for all}\ s\in [0,1].$$
  Therefore, $$\frac{\d}{\d s} V_t(s)=\frac{\d}{\d s}\tilde V_t(s),\
  i.e.\ v_t(s)^{-1}\dot v_t(s)=\tilde v_t(s)^{-1}\dot{\tilde
  v}_t(s)=:k_t(s).$$ Since the solution of
  $$d_s v_t(s)=v_t(s)k_t(s) d s,\quad v_t(1)=e$$ is unique, it
  follows that $\tilde v_t(s)=v_t(s)$ for all $s\in [0,1]$. In
  particular, $\tilde{v}_t(0)=v_t(0)$. The proof is complete.
  \end{proof}

  \noindent\textbf{Proof of Theorem \ref{t1.1} for $p=2$.}
  We have shown in (\ref{2.4}) that
  $$W_2(\nu,\sigma)^2=\int_{\pc(G)}
  \phi(\gamma_1)\,\d\nu+\int_{\pc(G)}\psi(\gamma_2)\,\d\sigma.$$
  By Lemma \ref{2.2}, $\phi$ is in $\mathbf{D}_1^2(\mu)$. So $\phi$
  is $\mu$-almost everywhere differentiable, so does also with
  respect to $\nu$ by the absolute continuity of $\nu$ relative to
  $\mu$. Since $d_{L^2}(\cdot,\cdot)$ is continuous from $\pc(G)\times \pc(G)$ to $ \R$,
  and $\C(\nu,\sigma)$ is tight, there exists an optimal transport plan $\pi\in \C(\nu,\sigma)$
  such that
  $$ W_2(\nu,\sigma)^2=\int_{\pc(G)\times\pc(G)} \dd{\gamma_1}{\gamma_2}^2\,\pi(\d\gamma_1,\d\gamma_2).$$
  Hence,
  $$\int_{\pc(G)\times\pc(G)}\phi(\gamma_1)+\psi(\gamma_2)\,\pi(\d\gamma_1,\d\gamma_2)
  =\int_{\pc(G)\times\pc(G)}\dd{\gamma_1}{\gamma_2}^2\,\pi(\d\gamma_1,\d\gamma_2).$$
  As $\phi=\psi^c$, there exists a measurable set
  $\Omega_1\subset\pc(G)\times \pc(G)$ such that $\pi(\Omega_1)=1$, and
  $$\phi(\gamma_1)+\psi(\gamma_2)=\dd{\gamma_1}{\gamma_2}^2,\quad
  \forall\,(\gamma_1,\gamma_2)\in\Omega_1.
  $$
  Since $\phi$ is $\nu$-a.e. differentiable, there exists a
  measurable set $A\subset \pc(G)$ with $\nu(A)=1$ on which $\phi$
  is differentiable everywhere. Let
  $\Omega=\Omega_1\cap(A\times\pc(G))$, then $\pi(\Omega)=1$.

  For a point $(\gamma_1,\gamma_2)\in \Omega$, Proposition \ref{2.3} yields
  that $\gamma_2\in \pc(G)$ is uniquely determined by $\gamma_1$ and $\phi$ such that
  $$\phi(\gamma_1)+\psi(\gamma_2)=\dd{\gamma_1}{\gamma_2}^2.$$
  Denote this map by $\gamma_2=\mathscr T(\gamma_1)$.
  Assume $\mathscr T$ is measurable, then for any measurable
  function $F$ on $\pc(G)\times\pc(G)$,
  \begin{align*}
  \int_{\pc(G)\times\pc(G)}F(\gamma_1,\gamma_2)\,\pi(\d\gamma_1,\d\gamma_2)&=
  \int_{\pc(G)\times\pc(G)}F(\gamma_1,\mathscr{T}(\gamma_1))\,\pi(\d\gamma_1,\d\gamma_2)\\
  &=\int_{\pc(G)}F(\gamma_1,\mathscr{T}(\gamma_1))\,\nu(\d\gamma_1).
  \end{align*} This implies that
  \begin{equation}\label{2.16}\pi=(id\times\mathscr T)_\ast \nu\quad \text{and}\ (\mathscr T)_\ast \nu=\sigma.\end{equation}

  If there exists another measurable map $\mathscr S:\pc(G)\ra \pc(G)$ such that $(\mathscr S)_\ast \nu=\sigma$ and
  $$W_2(\nu,\sigma)^2=\int_{\pc(G)}\dd{\gamma}{\mathscr S(\gamma)}^2\,\nu(\d\gamma).$$
  Then the measure $\tilde\pi:=(id\times \mathscr S)_\ast \nu$ is an optimal transport map. Since in above discussion  $\pi$ is arbitrary
  optimal transport plan in $\C(\nu,\sigma)$, applying (\ref{2.16}) to $\tilde\pi$, we obtain
  $$\tilde\pi=(id\times\mathscr T)_\ast \nu,\quad \text{and} \ \mathscr S=\mathscr T,\ \text{$\nu$-a.e.}.$$
  This proves the uniqueness of $\mathscr T$.

  Now we proceed to the measurability of $\mathscr T $.

  Let $\{\beta_n,\,n\geq 1\}\subset C^\infty([0,1],\R)$ be an
  orthonormal basis of the space $H(\R)=\big\{f:[0,1]\ra \R;\
  f(0)=0,\, \int_0^1|f'(s)|^2\d s<+\infty\big\}$. Define
  $$c_n(t)=\int_0^t\beta_n(s)\,\d s-t\int_0^1\beta_n(s)\,\d s.$$
  Let $\{e_1,\ldots,e_d\}$ be an orthonormal basis of $\G$. Then
  $\{\beta_n e_i,\, n\geq 1,\,i=1,\ldots,d\}$ be an orthonormal
  basis of $H(\G)$. Let $U_t(u)=\int_0^u V_t(s)\,\d s$. Then
  (\ref{2.15}) can be rewritten as
  \begin{equation*}
  \la V_t(1),c_n'(1)e_i\raa_\G=\int_0^1\la \dot{U}_t(s),\beta_n'(s)
  e_i\raa_\G\,\d s=\la U_t,\beta_n e_i\raa_{H(\R)}.
  \end{equation*}
  It follows that
  \begin{equation}\label{2.17}
  U_t(s)=\sum_{n\geq 1}\sum_{i=1}^d \la V_t(1),c_n'(1)e_i\raa_\G
  \beta_n(s)e_i.
  \end{equation}
  We have shown in the proof of Proposition \ref{2.3} that $V_t(1)$ is measurable
  with respect to $\gamma_1$ if $V_t(1)$ is continuous at $t$. So
  for $t\not \in \Omega$, $U_t(s)$ is also measurable with respect
  to $\gamma_1$ for each $s\in[0,1]$, so does $V_t(s)$. Then by the
  definition (\ref{2.12}), $$ d_sv_t(s)=v_t(s)dV_t(s),\quad
  v_t(1)=e.$$
  Therefore for each $t\in [0,1]\backslash \Omega$, $v_t(s)$ is a
  measurable mapping of $\gamma_1$ for each $s\in [0,1]$. Then we
  obtain the measurability of $\gamma_1\mapsto
  \gamma_2(t)=\gamma_1(t)v_t(0)^{-1}$ for $t\in [0,1]\backslash \Omega$.

  Take a subdivision  $\mathcal P=\{0<1/N<\cdots<(N-1)/N<1\}$ of [0,1]. We can take $N$
  large enough so that for each $i=0,\ldots,N-1$,
  $\gamma_2(i/N)$ and $\gamma_2((i+1)/N)$ are not in the cut locus of
  each other. Define a continuous curve
  $\gamma^{(N)}(t)=\gamma_2(t)$ for $t\in\mathcal P$ and connect
  $\gamma^{(N)}(i/N)$ with $\gamma^{(N)}((i+1)/N)$ by the unique
  minimizing geodesic. $\gamma^{(N)}$ is continuous, and $\gamma_1\mapsto\gamma^{(N)}$
  is measurable due to the measurability of the solution of geodesic equation.
  Letting $N$ tend to $+\infty$, $\gamma^{(N)}$ converges
  uniformly to $\gamma_2$, so $\gamma_1\mapsto\gamma_2=\mathscr T(\gamma_1)$
  is also measurable. Therefore, we have shown the measurability
  of the map $\mathscr T$, which shows the existence and uniqueness of optimal transport map in Theorem \ref{t1.1}.

  To complete
  the proof of this theorem, it remains to prove the explicit expression of the optimal transport map $\mathscr T$. As in Proposition \ref{t2.3}, for each $t\in [0,1]$, there exists a constant geodesic $(v_t(s))_{s\in[0,1]}$ on $G$ connecting $\gamma_2(t)^{-1}\gamma_1(t)$ to $e$. For any given continuously differentiable function $c:[0,1]\ra \R$ with $c(0)=c(1)=0$, define $\hat v_{t,\veps}(s)=v_t(s)e^{c(s)\veps h(t)}$ for $h\in H$ and $\veps\in \R$. Then $\hat v_{t,\veps}(0)=v_t(0)$, $\hat v_{t,\veps}(1)=v_t(1)$. Thus the function $\veps\mapsto L(\hat v_{t,\veps})^2$ attains its minimum value at $\veps=0$, which yields that
  \[\frac{\d}{\d \veps}\Big|_{\veps=0}L(\hat v_{t,\veps})^2= \int_0^1 \la v_t^{-1}(s)\dot{v}_t(s),c'(s)h(t)\raa_{\G}\d s=0.\]
  Taking $c(s)=\sin(k\pi s)$ for $k\in \Z$ and $h(t)= a\in \G$ in previous equation,  we obtain
  \begin{equation}\label{e-1.1}
  \int_0^1\la v_t^{-1}(s)\dot{v}_t(s),\cos(k\pi s)a\raa_{\G}\d s=0.
  \end{equation}
  The  arbitrariness of $k\in \Z$ and $a\in\G$ means that $s\mapsto v_t^{-1}(s)\dot{v}_t(s)$ is a constant function over $[0,1]$. According to \eqref{2.12} and \eqref{2.13}, if $\phi$ is differentiable at $\gamma_1$, then
  \[V_t(1)=\int_0^1v_t^{-1}(s)\dot{v}_t(s)\d s=v_t^{-1}(0)\dot{v}_t(0),\] and
  \begin{equation} \label{e-1.4}
  \int_s^1 V_t(1)\d t=-\frac 12 \frac{\d}{\d s}\big(\nabla \phi(\gamma_1)\big)(s).
  \end{equation}
  Since $t\mapsto V_t(1)$ is continuous on $[0,1]\backslash \Omega$ as shown in Lemma \ref{t2.3}, \eqref{e-1.4} yields that
  \begin{equation}\label{e-1.5}
  V_t(1)=\frac{1}{2}\frac{\d ^2}{\d t^2}\big(\nabla \phi(\gamma_1)\big)(t),\quad \text{$t\in [0,1]\backslash \Omega$}.
  \end{equation}
  As $s\mapsto \gamma_2(t)v_t(s)$ is a geodesic connecting $\gamma_1(t)$ to $\gamma_2(t)$, it can be expressed in terms of geodesic exponential map as
  \begin{equation}\label{e-1.2}
  \gamma_2(t)=\exp_{\gamma_1(t)}\Big(\ell_{\gamma_1(t)}v_t(0)^{-1}\dot{v}_t(0)\Big)
  =\exp_{\gamma_1(t)}\Big(\frac12\ell_{\gamma_1(t)}\frac{\d^2}{\d t^2}\big(\nabla \phi(\gamma_1)\big)(t)\Big).
  \end{equation}
  By Lemma \ref{t2.2}, $\phi$ is in $\mathbf D_1^2(\mu)$ and is $\mu$-almost surely differentiable. The expression \eqref{e-1.2} means that for $\mu$-a.e. $\gamma\in \pc(G)$
  \begin{equation}\label{e-1.3}
  \mathscr T (\gamma)(t)=\exp_{\gamma(t)}\Big(\frac1 2\ell_{\gamma(t)}\frac{\d^2}{\d t^2}\big(\nabla \phi(\gamma)\big)(t)\Big),\quad t\in[0,1]\backslash\Omega.
  \end{equation}
  We have completed the proof till now.\hfill\fin

\section{Proof of main result: general case $p>1$, $p\neq 2$}

  Now we shall prove Theorem \ref{1.1} for general $p>1$, $p\neq 2$. Recall that for two probability measures
  $\nu$ and $\sigma$
  on $\pc(G)$, define the $L^p$-Wasserstein distance between them
  by:
  \begin{equation}\label{2.18}
  W_p(\nu,\sigma)=\inf_{\pi}\Big\{\int_{\pc(G)\times\pc(G)}
  \dd{\gamma_1}{\gamma_2}^p\,\pi(\d\gamma_1,\d\gamma_2)\Big\}^{1/p},
  \end{equation} where the infimum is taken over $\C(\nu,\sigma)$.
  The difficulty in the case $p>1$ and $p\neq2$ is to prove the
  uniqueness of $\gamma_2$ by the equation
  $$\phi(\gamma_1)+\psi(\gamma_2)=\dd{\gamma_1}{\gamma_2}^p.$$
  We get around this difficulty by using a  more delicate
  variational method than the method used in the proof of Proposition \ref{t2.3}.

  \noindent\textbf{Proof of Theorem 1.1 for $p>1$ and $p\neq 2$.}
  According to Theorem \ref{t2.1}, there exists  a couple of
  functions
  $\phi$ and $\psi$ on $\pc(G)$ such that $\phi=\psi^c$, where the
  function $c(\gamma_1,\gamma_2)=\dd{\gamma_1}{\gamma_2}^p$ now. The
  boundedness of $d_{L^2}$ yields easily that $\phi=\psi^c$ is
  $d_{L^ 2}$-Lipschitz continuous, and hence belongs to $\mathbf
  D_1^2(\mu)$ thanks to Lemma \ref{t2.2}.

  To prove Theorem \ref{t1.1} for $p>1$, we can get along with the same
  lines as the proof for $p=2$. We omit similar steps in the argument, and only
  prove the main different part, which is
  to prove that: if it holds
  $$ \phi(\gamma_1)+\psi(\gamma_2)=\dd{\gamma_1}{\gamma_2}^p,$$
  where $p>1$, then $\gamma_2$ is uniquely determined by $\gamma_1$
  and $\phi$. In fact, let $v_t:[0,1]\ra G$ be a constant speed geodesic such that
  $v_t(0)=\gamma_2(t)^{-1}\gamma_1(t)$, $v_t(1)=e$ and
  $L(v_t)^2=\rho(\gamma_1(t),\gamma_2(t))^2$. Using the same variation
  as in the argument of Proposition \ref{t2.3} again, we can
  obtain
  \begin{equation}\label{2.19}
  \la \nabla
  \phi(\gamma_1),h\raa_H=-p\dd{\gamma_1}{\gamma_2}^{p-2}\int_0^1\la
  V_t(1),h(t)\raa_\G\,\d t
  \end{equation} instead of formula (\ref{2.13}). This yields that $
  \dd{\gamma_1}{\gamma_2}^{p-2} V_t(1)$ is uniquely determined by
  $\nabla \phi(\gamma_1)$ for almost every $t\in [0,1]$ .  Using the same
  variation as in the argument of Lemma \ref{2.4}, we get formula
  (\ref{2.15}) again
  $$\la V_t(1),c'(1) a\raa_\G=\int_0^1\la V_t(s),c''(s)a\raa_\G\,\d
  s,$$ for any $c\in C^2([0,1],\R)$ with $c(0)=c(1)=0$,  any
  $a\in\G$, and each $t\in[0,1]$. Taking $a=\dd{\gamma_1}{\gamma_2}^{p-2} b$ for $b\in
  \G$, we get
  \begin{equation}\label{2.20}
  \la \dd{\gamma_1}{\gamma_2}^{p-2}V_t(1),c'(1)b\raa_\G=\int_0^1\la
  \dd{\gamma_1}{\gamma_2}^{p-2} V_t(s),c''(s)b\raa_\G\,\d s.
  \end{equation}
  Assume $\tilde\gamma_2\in \pc(G)$
  such that
  $$\phi(\gamma_1)+\psi(\tilde\gamma_2)=\dd{\gamma_1}{\gamma_2}^p.$$
  Let $\tilde v_t:[0,1]\ra G$ be a minimizing geodesic such that
  $\tilde v_t(1)=e$ and $\tilde
  v_t(0)=\tilde\gamma_2(t)^{-1}\gamma_1(t)$. Analogously, define
  $\tilde V_t(u)=\int_0^u\tilde v_t(s)^{-1}\dot{\tilde v}_t(s)\,\d
  s$ and it holds that
  \begin{equation}\label{2.21}
  \la \dd{\gamma_1}{\tilde\gamma_2}^{p-2}\tilde
  V_t(1),c'(1)b\raa_\G=\int_0^1\la
  \dd{\gamma_1}{\tilde\gamma_2}^{p-2}\tilde
  V_t(s),c''(s)b\raa_\G\,\d s.
  \end{equation} Since $\dd{\gamma_1}{\tilde\gamma_2}^{p-2}\tilde
  V_t(1)$ is also determined by $\nabla \phi(\gamma_1)$ for almost
  everywhere $t\in[0,1]$, there exists a subset $\bar\Omega\subset[0,1]$ with full
  Lebesgue measure in $[0,1]$ such that
  \begin{equation}\label{2.22}
  \dd{\gamma_1}{\gamma_2}^{p-2}
  V_t(1)=\dd{\gamma_1}{\tilde\gamma_2}^{p-2}\tilde
  V_t(1),\quad \forall \,t\in \bar\Omega.
  \end{equation}
  Due to the denseness of functions in the form $s\mapsto c''(s) b$ in
  $L^2(\mu;\G)$, and the continuity of $s\mapsto V_t(s)$ and $s\mapsto \tilde V_t(s)$, we get
  \begin{equation}\label{2.23}
  \dd{\gamma_1}{\gamma_2}^{p-2}V_t(s)=\dd{\gamma_1}{\tilde\gamma_2}^{p-2}\tilde
  V_t(s), \quad \forall\, s\in [0,1],\ t\in \bar\Omega.
  \end{equation}
  It follows then
  \begin{equation*}
  \dd{\gamma_1}{\gamma_2}^{p-2}\dot{V}_t(s)=\dd{\gamma_1}{\tilde\gamma_2}^{p-2}\dot{\tilde
  V}_t(s), \quad \text{for a.e.}\, s\in [0,1], \ t\in \bar \Omega,
  \end{equation*} where dot $\cdot$ denotes the derivative relative
  to $s$. Since $v_t(s)$ and $\tilde v_t(s)$ are both minimizing
  geodesics, integrating both sides of previous equation over $s$ from 0 to 1 yields that
  \begin{equation}\label{2.24}
  \dd{\gamma_1}{\gamma_2}^{p-2}\rho(\gamma_1(t),\gamma_2(t))
  =\dd{\gamma_1}{\tilde\gamma_2}^{p-2}\rho(\gamma_1(t),\tilde\gamma_2(t)),\ \
  \forall\,t\in \bar\Omega.
  \end{equation} Then, integrating the square of both sides over $t$ from 0 to 1 yields
  \begin{equation}\label{2.25}
  \dd{\gamma_1}{\gamma_2}^{2(p-1)}=\dd{\gamma_1}{\tilde\gamma_2}^{2(p-1)}.
  \end{equation}
  Combining this with (\ref{2.22}), we obtain
  $$V_t(1)=\tilde V_t(1),\quad \forall \,t\in \bar\Omega.$$ Using Lemma
  \ref{t2.4}, we have $\gamma_2(t)=\tilde\gamma_2(t)$ for $t\in
  \bar\Omega$. The continuity of $\gamma_2(t)$ and $\tilde\gamma_2(t)$
  yields $\gamma_2(t)\equiv \tilde\gamma_2(t)$ for $t\in [0,1]$,
  and hence $\gamma_2\in \pc(G)$ is uniquely determined.
  \hfill \fin

  \noindent\textbf{Proof of Theorem \ref{t1.3}.}
  We only prove the assertion of this theorem for path groups, and the corresponding assertion for loop groups can be proved in the same way.

  For $\nu_0,\,\nu_1\in \p_0(\pc(G))$ with $\nu_0$ being absolutely continuous w.r.t. the Wiener measure $\mu$, according to Theorem \ref{t1.1}, there exists a unique optimal map $\mathscr T:\pc(G)\ra \pc(G)$ such that
  \[\pi_0:=(id\times \mathscr T)_{*}\nu_0\]
  attains the $L^2$-Wasserstein distance between $\nu_0$ and $\nu_1$, i.e.
  \[W_2(\nu_0,\nu_1)^2=\int_{\pc(G)\times\pc(G)}\!\!\!d_{L^2}(\gamma_1,\gamma_2)^2\pi_0(\d\gamma_1,\gamma_2).\]
  Let $\phi$ and $\psi$ be the Kantorovich potentials, then
  \[W_2(\nu_0,\nu_1)^2=\int_{\pc(G)\times\pc(G)}\!\!\!\!\big(\phi(\gamma_1)+\psi(\gamma_2)\big)\pi_0(\d\gamma_1,\d \gamma_2).\]
  Then the support of $\pi_0$ is clearly located in the set
  \[A=\{(\gamma_1,\gamma_2)\in \pc(G)\times\pc(G);\ \phi(\gamma_1)+\psi(\gamma_2)=d_{L^2}(\gamma_1,\gamma_2)^2\}.\]
  For $\gamma_1,\,\gamma_2$ satisfying
  \[\phi(\gamma_1)+\psi(\gamma_2)=d_{L^2}(\gamma_1,\gamma_2)^2,\]
  in the proof of Proposition \ref{t2.3}, we have shown that for each $t\in [0,1]$ there exists a constant speed geodesic $v_t:[0,1]\ra G$ such that $v_t(0)=\gamma_2(t)^{-1}\gamma_1(t)$, $v_t(1)=e$, and
  \[L(v_t)^2=\int_0^1\big|v_t^{-1}(s)\frac{\d}{\d s}v_t(s)\big|_{\G}^2\d s=\rho(\gamma_1(t),\gamma_2(t))^2.\]
  Set $u_t(s)=\gamma_2(t)v_t(s)$ for $s\in[0,1]$, then $u_t(0)=\gamma_1(t)$ and $u_t(1)=\gamma_2(t)$.
  For any given $\lambda\in [0,1]$, let $u^\lambda_\cdot$ be in $\pc(G)$ defined by $u^\lambda_t=u_t(\lambda)$ for $t\in [0,1]$. The distance between $\gamma_1$ and $u^\lambda$ is
  \begin{equation}\label{geo}d_{L^2}(\gamma_1,u^\lambda)=\lambda d_{L^2}(\gamma_1,\gamma_2).
  \end{equation}
  Indeed, the curves $s\mapsto u_t(\lambda s)$ and $s\mapsto u_t(\lambda+(1-\lambda)s)$ connect respectively $\gamma_1(t)$ to $u^\lambda_t$ and $u^\lambda_t$ to $\gamma_2(t)$. Then
  \[\rho(\gamma_1(t),u^\lambda_t)\leq \Big(\int_0^1\lambda^2\big|v_t^{-1}(\lambda s)\frac{\d}{\d s}v_t(\lambda s)\big|^2\d s\Big)^{1/2}=\lambda \rho(\gamma_1(t),\gamma_2(t)).\]
  Similarly, $\rho(u^\lambda_t,\gamma_2(t))\leq (1-\lambda)\rho(\gamma_1(t),\gamma_2(t))$. Together with the triangle inequality, we can get $\rho(\gamma_1(t),u^\lambda_t)=\lambda \rho(\gamma_1(t),\gamma_2(t))$ for all $\lambda\in [0,1]$, and further \eqref{geo} holds. Consequently, $\lambda\mapsto u^\lambda$ is a geodesic in $(\pc(G),d_{L^2})$ connecting $\gamma_1$ to $\gamma_2$.

  Set $\Phi_\lambda(\gamma_1)=u_{\cdot}^\lambda$ and $\nu_\lambda=( \Phi_\lambda)_*\nu_0$ for $\lambda\in[0,1]$. Then
  \[W_2(\nu_0,\nu_\lambda)\leq \Big(\int_{\pc(G)\times\pc(G)}\!\!\! d_{L^2}(\gamma_1,
  u^\lambda)^2\d \nu_0(\gamma_1)\Big)^{1/2}=\lambda W_2(\nu_0,\nu_1),\]
  and
  \[W_2(\nu_\lambda,\nu_1)\leq \Big(\int_{\pc(G)\times\pc(G)}\!\!\! d_{L^2}(u^\lambda,\gamma_2)^2\d \nu_0(\gamma_1)\Big)^{1/2}=(1-\lambda)W_2(\nu_0,\nu_1).\]
  By the triangle inequality, it holds
  \[W_2(\nu_0,\nu_\lambda)=\lambda(\nu_0,\nu_1),\ \ W_2(\nu_\lambda,\nu_1)=(1-\lambda)W_2(\nu_0,\nu_1).\]
  Hence, $\nu_\lambda$  for $\lambda\in [0,1]$ is a geodesic in $\p(\pc(G))$ w.r.t. the Wasserstein distance $W_2$ connecting $\nu_0$ to $\nu_1$. The proof is complete. \hfill\fin

  \section{Optimal transport map on loop groups}

   Let $G$ be a connected compact Lie group and its Lie algebra $\G$ is
   endowed with an $\Ad$-invariant metric $\la\,,\raa_\G$.
   Let $$\lg=\{\ell:[0,1]\ra G\,\text{continuous}; \,
   \ell(0)=\ell(1)=e\}.$$
   The product in $\lg$ is defined pointwisely by
   $(\ell_1\cdot\ell_2)(\theta)=\ell_1(\theta)\cdot\ell_2(\theta)$, $\theta\in
   [0,1]$. With the uniform topology
   $$d_\infty(\ell_1,\ell_2)=\sup_{\theta\in
   [0,1]}\rho(\ell_1(\theta),\ell_2(\theta)),$$ where $\rho$ is
   Riemannian distance on $G$, $\lg$ becomes a topological group. Recall
   that
   $$H(\G)=\Big\{h:[0,1]\ra \G;\,h(0)=0,\,|h|_H^2=\int_0^1\big|\dot{h}(t)\big|_\G^2\,\d
  t<+\infty\Big\}.$$
   Let $$H_0(\G)=\big\{h\in H(\G);\,h(0)=h(1)=0\big\}.$$ For $h\in H_0(\G)$,
   set $\displaystyle |h|_{H_0}=\Big(\int_0^1|\dot
   h(\theta)|_\G^2\,\d\theta\Big)^{1/2}$. It has been shown in
   \cite{Mal} that there is a Brownian motion $(g(t))$ on $\lg$.
   In order to be consistent in notations as convention, in the sequel, we shall fix
   $\nu$ to be the law of Brownian motion $g(1)$ on $\lg$, which is
   called heat kernel measure. Let $\mu_0$ denote the pinned Wiener measure on $\lg$.
   Due to \cite{DS}, $\nu$ is absolutely
   continuous with respect to the pinned Wiener measure $\mu_0$.
   According to \cite{AD}, $\mu_0$ is also absolutely continuous with
   respect to heat kernel measure $\nu$.

   For a cylindrical function
   $F:\lg\ra \R$ in the form
   $$F(\ell)=f(\ell(\theta_1),\ldots,\ell(\theta_n)),\quad f\in
   C^\infty(G^n),$$ and $h\in H_0(\G)$, define
   $$(D_h F)(\ell)=\frac{\d}{\d \veps}\Big|_{\veps=0}F(\ell e^{\veps
   h})=\sum_{i=1}^n\la \partial_i
   f,\ell(\theta_i)h(\theta_i)\raa_{T_{\ell(\theta_i)}G},$$ where
   $\partial_i f$ denotes the $i$th partial derivative. The gradient
   operator $\nabla^\lc$ on $\lg$ is defined as
   \begin{equation*}
   \big(\nabla^\lc
   F\big)(\ell)=\sum_{i=1}^n\ell^{-1}(\theta_i)(\partial_i
   f)G(\theta_i,\cdot),
   \end{equation*}
   where $G(\theta_i,\theta):=\theta_i\wedge \theta-\theta_i\theta$.
   Consider $$\mathscr E(F,F):=\int_{\lg}|\nabla^\lc
   F|_{H_0}^2\,\d\nu.$$ Then $\mathscr E$ defined on the set of cylindrical functions is closable, and let
   $\mathbf D_1^2(\nu)$ be the domain of the associated Dirichlet
   form.

   Now, we introduce several distance on $\lg$. Firstly, the
   $L^2$-distance is defined by:
   \begin{equation}\label{4.1}
   \dd{\ell_1}{\ell_2}=\Big(\int_0^1\rho(\ell_1(\theta),\ell_2(\theta))^2\,\d
   \theta\Big)^{1/2},\quad \ell_1,\,\ell_2\in \lg.
   \end{equation}
   Secondly, we shall recall the definition of Riemannian distance
   on $\lg$. In \cite{FS}, it has shown the existence and uniqueness of
   optimal transport map for the Monge-Kantorovich problem with the
   Wasserstein distance defined by the square of Riemannian distance
   on $\lg$.

   A continuous curve $\gamma:[0,1]\ra \lg$ is said to be admissible
   if there exists $z\in H(H_0)$ such that
   \begin{equation}\label{4.2}
   \frac{\partial }{\partial
   t}\gamma(t,\theta)=\gamma(t,\theta)\frac{\partial}{\partial t} z(t,\theta),\quad
   \gamma(0,\theta)=e.
   \end{equation} Here
   $$H(H_0)=\Big\{z:[0,1]\ra H_0(\G);\  z_t=\int_0^t\frac{\partial}{\partial
   s}z(s)\,\d s, \|z\|^2:=\int_0^1|\frac{\partial}{\partial
   s}z(s)|_{H_0}^2\,\d s<+\infty\Big\}.$$ For a continuous curve $\gamma$ on
   $\lg$, if it is admissible, its length is defined by
   $$L(\gamma)=\Big(\int_0^1|\frac{\partial}{\partial s}
   z(s)|_{H_0}^2\,\d s\Big)^{1/2};$$ otherwise, its length
   $L(\gamma)=+\infty$.
   The Riemannian distance $d_L$ on $\lg$ is defined by
   \begin{equation}\label{4.3}
   d_L(\ell_1,\ell_2)=\inf\big\{L(\gamma);\
   \gamma(0)=\ell_1,\,\gamma(1)=\ell_2\big\},
   \end{equation}
   where $\gamma$ runs over the set of all continuous curves on
   $\lg$. It is clear that $d_L$ is left invariant:
   $d_L(\ell\ell_1,\ell\ell_2)=d_L(\ell_1,\ell_2)$,
   $\ell,\,\ell_1,\,\ell_2\in \lg$. It has been shown in \cite[Proposition 3.4]{Sh}
   that for $\ell_1,\,\ell_2\in\lg$, $d_\pc(\ell_1,\ell_2)\leq
   d_L(\ell_1,\ell_2)$. Therefore, it holds
   \begin{equation}\label{4.4}
   d_{L^2}(\ell_1,\ell_2)\leq d_\infty(\ell_1,\ell_2)\leq
   d_\pc(\ell_1,\ell_2)\leq d_L(\ell_1,\ell_2).
   \end{equation}
   According to the Rademacher's theorem \cite[Theorem 1.5]{Sh}, we get
   \begin{lem}\label{t4.1}
   Every $d_L$-Lipschitz (hence, $d_{L^2}$-Lipschitz) continuous
   function $F$ is in $\mathbf D_1^2(\nu)$.
   \end{lem}

   After these preparation, we are in a position to state our results.
   The results in Theorem \ref{t1.2} are parts of the results in the following two theorems.
   \begin{thm}\label{t4.2}
   For every probability measures $\sigma_1$ and $\sigma_2$ on
   $\lg$. Assume $\sigma_1$ is absolutely continuous with respect to
   the heat kernel measure $\nu$ on $\lg$. Then for each $p>1$,
   there exists a unique measurable map $\mathscr T_p:\lg\ra \lg$ such that
   it pushes $\sigma_1$ forward to $\sigma_2$ and
   $$W_{p,d_{L^2}}(\sigma_1,\sigma_2)^p=\int_{\lg}\dd{\ell}{\mathscr
   T_p(\ell)}^p\,\sigma_1(\d\ell),$$
   where
   $$W_{p,d_{L^2}}(\sigma_1,\sigma_2)^p:=
   \inf\Big\{\int_{\lg\times\lg}\dd{\ell_1}{\ell_2}^p\,\pi(\d\ell_1,\d\ell_2)\Big\},$$
   where the infimum runs over the set of all probability measures on $\lg\times \lg$ with marginals $\sigma_1$ and
   $\sigma_2$ respectively.
   \end{thm}
   \begin{proof}(\textbf{Sketched})\quad The proof of this theorem
   gets along the same lines as the proof of Theorem \ref{t1.1}.
   First, Theorem \ref{t2.1} guarantees the existence of the Kantorovich potential $\phi$
   and $\psi$ such that
   $$\phi(\ell)=\psi^c(\ell):=\inf_{\ell'\in \lg}\big\{\dd{\ell}{\ell'}^p-\psi(\ell')\big\}.$$
   Then $\phi$ is $d_{L^2}$-Lipschitz continuous.
   By the Rademacher's theorem, Lemma \ref{t4.1}, $\phi$ belongs to $\mathbf{D}_1^2(\nu)$.
   Then using the variational method to show the uniqueness of
   $\ell_2\in \lg$ such that
   $$\phi(\ell_1)+\psi(\ell_2)=\dd{\ell_1}{\ell_2}^p,$$ if $\phi$ is differentiable at $\ell_1$.
   This progress is completely similar to the proof of Proposition \ref{t2.3} for case $p=2$
   and the discussion in section 3 for case $p>1$.
   In this step, the different point is just to replace $h\in H(\G)$ with $h\in
   H_0(\G)$. Then the desired map is the map defined by $\mathscr
   T_p(\ell_1)=\ell_2$ such that above equation holds. The
   measurability of this map comes from the construction as in the
   argument of Theorem \ref{t1.1}.
   \end{proof}

   \begin{thm}\label{t4.3}
   On $\lg$, for each $p>1$, there exists a unique measurable map $\mathscr
   T_p:\lg\ra \lg$ such that $\mathscr T_p$ pushes heat kernel
   measure $\nu$ forward to pinned Wiener measure $\mu_0$ such that
   \begin{equation}\label{4.5}
   W_{p,d_{L^2}}(\nu,\mu_0)^p=\int_{\lg}\dd{\ell}{\mathscr
   T_p(\ell)}^p\,\d\nu(\ell).
   \end{equation} Moreover, $\mathscr T_p$ is $\nu$-a.e. reversible, and its inverse $\mathscr T_p^{-1}$
   pushes $\mu_0$ forward to $\nu$.
   \end{thm}

   \begin{proof}
   Noting that $\mu_0$ and $\nu$ is mutually absolutely continuous
   with respect to each other, applying Theorem \ref{t4.2} yields that there exists a
   measurable map $\mathscr T_p:\lg\ra \lg$ which pushes $\nu$ forward to $\mu_0$
   and a measurable map $\mathscr S_p:\lg\ra \lg$ which pushes $\mu_0$ forward to $\nu$. Furthermore,
   \begin{equation}\label{4.6}
   W_{p,d_{L^2}}(\nu,\mu_0)^p=\int_{\lg}\dd{\ell}{\mathscr
   T_p(\ell)}\,\d\nu(\ell)=\int_{\lg}\dd{\mathscr
   S_p(\ell)}{\ell}^p\,\d\mu_0(\ell).
   \end{equation}
   For any measurable function $F$ on $\lg$, we have
   \begin{equation}\label{4.6}
   \int_{\lg}F(\ell)\,\d\nu(\ell)=\int_{\lg}F(\mathscr
   S_p(\ell))\,\d\mu_0(\ell)=\int_{\lg}F(\mathscr S_p(\mathscr
   T_p(\ell)))\,\d\nu(\ell).
   \end{equation} Therefore,
   $$\mathscr S_p\circ \mathscr T_p=id,\quad \nu\text{-}a.e.,$$
   where $id$ denotes the identity map. We conclude the argument immediately.
   \end{proof}

\noindent \textbf{Acknowledgement} This work is supported by NSFC (No. 11431014, 11771327). The author is grateful to Professor S.Z. Fang for the invitation of visiting to University of Bourgogne in July 2017 and for valuable discussion on this topic.

\end{document}